\documentclass[12pt,a4paper,reqno]{amsart}
\usepackage[english]{babel}
\usepackage[utf8]{inputenc}
\usepackage{amsmath,amsfonts,amssymb,amsthm}
\usepackage{a4wide}
\usepackage{hyperref}
\usepackage{graphicx}
\usepackage{xfrac}
\usepackage{color}

\newtheorem{thm}{Theorem}[section]
\newtheorem{lem}[thm]{Lemma}
\newtheorem{cor}[thm]{Corollary}

\newtheorem{prop}[thm]{Proposition}
\theoremstyle{definition}
\newtheorem{definition}[thm]{Definition}
\newtheorem{exam}[thm]{Example}

\theoremstyle{remark}
\newtheorem{rem}[thm]{Remark}

\newcommand{\bs}{\boldsymbol}

\DeclareMathOperator{\dens}{dens}

\newcommand{\N}{\mathbb{N}}

\newcommand{\R}{\mathbb{R}}

\newcommand{\Z}{\mathbb{Z}}

\renewcommand{\phi}{\varphi}
\renewcommand{\rho}{\varrho}
\renewcommand{\epsilon}{\varepsilon}
\renewcommand{\geq}{\geqslant}
\renewcommand{\leq}{\leqslant}

\begin{document}

\title[Catalan numbers and discrepancy]{Catalan numbers as discrepancies for a family of substitutions on infinite alphabets}

\author{Dirk Frettl\"oh}
\address{Technische Fakult\"at, Bielefeld University,\newline\hspace*{\parindent}Postfach 100131, 33501 Bielefeld, Germany}
\email{dfrettloeh@techfak.de}

\author{Alexey~Garber}
\address{
School of Mathematical \& Statistical Sciences,
\newline\hspace*{\parindent}The University of Texas Rio Grande Valley, \newline\hspace*{\parindent}1 West University Blvd., Brownsville, TX 78520, USA.}
\email{alexey.garber@utrgv.edu}

\author{Neil Mañibo}
\address{Fakultät für Mathematik,  Bielefeld University,\newline\hspace*{\parindent}Postfach 100131, 33501 Bielefeld, Germany}
\address{ School of Mathematics and Statistics,\newline\hspace*{\parindent}Open University, Walton Hall, Kents Hill, 
\newline\hspace*{\parindent}Milton Keynes, United Kingdom MK7 6AA}
\email{cmanibo@math.uni-bielefeld.de, neil.manibo@open.ac.uk}

\date{\today}
\subjclass{05B45, 37B10, 52C23}
\keywords{Substitutions, Delone set, infinite alphabets, discrepancy, Catalan numbers, aperiodic tilings}

\begin{abstract}
In this work, we consider a class of substitutions on infinite alphabets and
show that they exhibit a growth behaviour which is impossible for substitutions on finite alphabets. While for both settings the leading term of the tile counting function is exponential (and guided by the inflation factor), the behaviour of the second-order term  is strikingly different. For the finite setting, it is known that the second term is also exponential or exponential times a polynomial. 
We exhibit a large family of examples  where the second term is at least exponential in $n$ divided by half-integer powers of $n$, where $n$ is the number of substitution steps. In particular, we provide an identity for this discrepancy in terms of linear combinations of Catalan numbers.
\end{abstract}

\maketitle

\begin{center}
\emph{Dedicated to our dear colleague and good friend, Uwe.}
\end{center}

\section{Introduction} \label{sec:intro}

Exponential growth or decay appears in many discrete or continuous dynamical systems. A perfect example of such a behaviour is the famous Fibonacci sequence that can be used to describe a (very simplified) growth of a population of bunnies. The Fibonacci sequence grows like $\Theta(\phi^n)$ where $\phi=\frac{\sqrt5+1}{2}$ is the golden mean; see Section \ref{sec:notations} for the notations we use for asymptotic behaviour.

In the sequel, we assume that the reader is familiar with basic terms and concepts of aperiodic order. Rather than explaining those basic terms here we will refer to the monograph \cite{BG} when necessary.

For (classical) primitive substitution tilings \cite{BG}, the asymptotic growth of many quantities scales like $\lambda^n$, where $\lambda$ is the associated inflation 
factor\,---\,respectively its $d$-th power, if the tiling lives in $\R^d$, for $d \ge 2$. For example, the number of tiles of all types (or of any single tile type) in the $n$-th iteration of the substitution applied to one prototile grows like $\Theta(\lambda^n)$. Both these phenomena originate from the uniqueness of the Perron--Frobenius eigenvalue $\lambda$ of the substitution matrix $M$ and from the Jordan form of $M$ and its powers $M^n$.

In the context of primitive substitutions, the {\it discrepancy} is the difference between the tracked quantity $q(n)$ and the expected value $c\cdot \lambda^n$ for the appropriate constant $c>0$. Its asymptotics are determined by the second largest (in absolute value) eigenvalue $\lambda'$. Again by using the Jordan form of $M^n$, this discrepancy can be written as $P(n)|\lambda'|^n$ for some polynomial $P$ if $\lambda'$ is real, and with some additional trigonometric factor if $\lambda'$ is complex.

The discrepancy plays a crucial role in the study of bounded-distance equivalence relations for point sets. Two Delone sets are bounded-distance equivalent to each other, if there is a bijection $f$ between them such that $|f(x)-x|<C$ for some $C>0$; see \cite{FGS,SS,Sol14} and references therein.  A similar notion of discrepancy lies in the core of the study of bounded remainder sets; see \cite{FG,HKK} and references therein.

This paper studies discrepancies for a certain class of substitutions on infinite alphabets, which belongs to a generalisation of primitive substitutions introduced in \cite{MRW,MRW2}. This particular class was studied in \cite{FGM}, where it was shown that in the most general setting it allows any real number greater than $2$ to be the inflation factor
of such a substitution.

More precisely, given a bounded sequence of non-negative integers ${\bs a}=(a_i)_{i\geq 0}$ satisfying certain assumptions, ${\bs a}$ defines a ``pre-substitution'' on the countable alphabet $\mathbb{N}_0\simeq \{[i]\}$; see Eq.~\eqref{eq:pre-subs}. 
This pre-substitution can be extended to a substitution $\varrho^{ }_{{\bs a}}$ on a compact alphabet $\mathcal A$ via a suitable embedding of $\mathbb{N}_0$ into a shift space; see \cite{FGM} for the complete construction. This substitution $\varrho^{ }_{{\bs a}}$ 
on a now compact alphabet $\mathcal{A}$
satisfies several properties that hold for primitive substitutions, including existence of letter frequencies and of an associated geometric substitution with inflation factor $\lambda$. We refer to \cite{FGM} and \cite{MRW,MRW2} for more details. We briefly remark that the tilings generated by these substitutions have infinite local complexity (ILC), both combinatorially and geometrically. 
One primary motivation for the series of works of the authors in the infinite alphabet regime is the development of an infinite-dimensional renormalisation scheme which would cover such objects, including the pinwheel tilings of the plane, which have since resisted a complete spectral characterisation; see \cite{BGF,GD} for some statistical and spectral properties of pinwheel tilings.

The question that inspired this paper was whether the Delone sets arising from these substitutions are bounded-distance equivalent to $\alpha \Z$ for some appropriate $\alpha$. 
Hence, our main focus is the {\it discrepancy function} $d_{\bs a}(n)$, which we define as follows. First, we count the number of letters in $\varrho^n_{\bs a}([0])$. Due to existence of frequencies and associated tile lengths, this quantity is $\sfrac{1}{c_{\bs{a}}}\cdot \lambda^n+d_{\bs a}(n)$, 
where $d_{\bs a}(n) \in o(\lambda^n)$. Here
\begin{itemize}
\item $\lambda$ is the inflation factor, and
\item $c_{\bs{a}}$ is the  average tile length, assuming the length of $[0]$ is 1 (hence $\sfrac{1}{c_{\bs{a}}}$ equals the density of the corresponding Delone set in $\mathbb{R}$, see \cite{BG}).
\end{itemize}
The paper is organised as follows. 
In Section~\ref{sec:def}, we provide necessary definitions and some background on substitutions on infinite alphabets, together with specific results for the class of substitutions we will be considering. In Section \ref{sec:catalan}, we prove that, if $a_i=1$ for all $i$, the discrepancy $d_{\bs a}$ is (a multiple of) the tail of the power series 
$$\frac{1-\sqrt{1-4x}}{2x}=\sum_{i=0}^\infty C_ix^i$$
evaluated at $x=\sfrac{4}{25}$. Here $C_k$ is the $k$-th Catalan number; see for instance \cite{Pak,Sta1,Sta2} or OEIS sequence \href{http://oeis.org/A000108}{A000108} \cite{oeis}. In particular, in Theorem \ref{thm:all1}, we show for this specific choice of ${\bs a}$  that $d_{\bs a}(n)$ belongs to $\Theta\left( \frac{2^n}{n^{\sfrac32}}\right)$. 

In Section \ref{sec:stab}, we tackle the more general case of eventually constant sequences ${\bs a}$. We show that, in that case, the discrepancy $d_{\bs a}(n)$ is also related to Catalan numbers through a non-homogeneous linear recurrence relation from Lemma \ref{lem:catlinear}. In Theorem \ref{thm:stab}, we show that there exists a non-negative integer $q$ such that some subsequence of $d_{\bs a}(n)$  grows at least as fast as $\Omega \left( \frac{2^n}{n^{q+\sfrac32}}\right)$. This implies that none of the Delone sets arising from the substitutions in this class is bounded distance equivalent to $\alpha \Z$, for any $\alpha>0$. This is Corollary~\ref{cor:bde}. By the results in \cite{SS} (see also \cite{FGS}), this in turn implies that each substitution $\rho_{\bs a}$ gives rise to uncountably many equivalence classes of Delone sets with respect to bounded distance equivalence. 

We want to emphasise that both asymptotics are completely different from the case of substitutions on finite alphabets where (for some subsequence) the similarly defined discrepancy is of order $\Theta\left(n^q|\lambda'|^n\right)$ for some non-negative integer $q$ and for some (possibly complex) number $\lambda'$. In Section \ref{sec:exam}, we apply the general framework from Section~\ref{sec:stab} to two concrete examples. Working through these examples, we show how the results from Section~\ref{sec:stab} are used to get a recurrence relation for the discrepancy function. Notably, we are able to get explicit formulas for the discrepancies in these two examples.

Section \ref{sec:moreexam} complements the above results by numerical computations of discrepancies of further examples, exploring the different phenomena that may occur for substitutions on infinite alphabets. Our examples study possible cases for some auxiliary parameter $\mu$ and discuss how the corresponding values of $\lambda$ change (or do not change) the growth rate of the discrepancy function.

\subsection{Notations for asymptotic behaviour} \label{sec:notations}

We use the following families for asymptotic behaviour \cite{K}. Here, we assume that $f$ and $g$ are functions of a non-negative integer parameter $n$ and that $g(n)> 0$ for sufficiently large $n$. 
\begin{itemize}
    \item $f\in o(g)$ if $\lim\limits_{n\rightarrow \infty}\frac{|f(n)|}{g(n)}=0$; 
    \item $f\in O(g)$ if $\limsup\limits_{n\rightarrow \infty}\frac{|f(n)|}{g(n)}$ is finite;
    \item $f\in \Omega(g)$ if $\liminf\limits_{n\rightarrow \infty}\frac{|f(n)|}{g(n)}>0$ or infinite;
    \item $f\in \Theta(g)$ if both $\limsup\limits_{n\rightarrow \infty}\frac{|f(n)|}{g(n)}$ and $\liminf\limits_{n\rightarrow \infty}\frac{|f(n)|}{g(n)}$ are positive numbers. 
\end{itemize}

\section{Preliminaries}
\label{sec:def}
\subsection{A class of substitutions over a compact alphabet}
Let $\mathcal{A}=\mathbb{N}_0\cup \left\{\infty\right\}$ be the one-point compactification of the set $\mathbb{N}_0$ of non-negative integers. 
We consider a class of substitutions on the alphabet $\mathcal{A}$ that maps letters in $\mathcal{A}$ to finite words over $\mathcal{A}$. The class is parameterised by a bounded sequence $\bs{a}=(a_i)_i$ of non-negative integers. 
We restrict to sequences $\bs{a}=(a_i)_i$ which are \emph{eventually constant}, i.e., there exists $k\in \mathbb{N}$ for which $a_i=a$ for some $a>0$ and for all 
$i\geqslant k$. We additionally require that $a_0>0$. The substitution associated to $\bs{a}$ is constructed as follows. We define the following ``pre-substitution'' on letters $[i]$ with $i\in \N_0$, see \cite[Def. 3.1]{FGM}, via
\begin{align}
\varrho^{ }_{\bs a}([0])&=[0]^{a_0}[1], \text{ and} \label{eq:pre-subs}\\
\varrho^{ }_{\bs a}([i])&=[0]^{a_i}[i-1][i+1] \; \text{ for }i>0. \nonumber
\end{align}
Since both $i-1$ and $i+1$ go to $\infty$ as $i\to \infty$, and since $a_i$ is eventually constant and is equal to $a$ for $i\geq k$, we then set 
\[
\varrho^{ }_{\bs a}([\infty])=[0]^{a}[\infty][\infty].
\]
This makes the rule $\varrho^{ }_{\bs{a}}\colon \mathcal{A}\to \mathcal{A}^{+}$ a continuous map, where the topology on the set $\mathcal{A}^{+}$ of finite words over $\mathcal A$ is the topology of a disjoint union; see \cite{MRW2}. Moreover, the restrictions we impose on ${\bs a}$ imply that $\varrho^{ }_{\bs{a}}$ is a \emph{primitive} substitution. In the compact alphabet setting, $\varrho_{\bs a}$ being primitive means that, for any non-empty open set $U\subset \mathcal{A}$, there exists a power $n:=n(U)$ such that $\varrho_{\bs a}^{n}(b)$ contains a letter in $U$, for all $b\in\mathcal{A}$; see \cite[Thm. 3.5]{FGM}. 
Note that such a rule results in a well-defined shift space $(X_{\varrho},S)$ that has many nice properties of primitive substitutions on finite alphabets (e.g., minimality, unique ergodicity); see \cite[Thm. 1.4]{FGM}. For more details on dynamical properties of substitutions on compact alphabets, we refer the reader to \cite{MRW2}.

Since we are interested in discrepancies for tile counting functions, we will focus on \emph{supertiles}, i.e., words of the form $\varrho^{n}_{\bs{a}}([i])$ and the associated substitution operator $M$, which is the generalisation of the substitution matrix for infinite alphabets. More specifically, we will be looking at the number of words in the $n$-th order supertile $\varrho^{n}_{\bs{a}}([0])$ of type $[0]$ as $n$ grows. 

\subsection{Substitution operator, natural length function, and frequencies}

Let $M$ be a bounded  linear operator on a Banach space $E$. The \emph{spectrum} of $M$, which we denote by $\sigma(M)$, is the set of all complex numbers $\lambda$ for which $(\lambda\bs{I}-M)$ is not invertible, where $\bs{I}$ denotes the identity operator. The \emph{spectral radius} of $M$ is defined as $r(M)=\sup\left\{|\lambda|\colon \lambda\in \sigma(M)\right\}$. Let $B_1(0)$ be the unit ball in $E$. The operator $M$ is called \emph{compact} if the image of $B_1(0)$ under $M$ is relatively compact. This is a rather strong condition. In fact, for the substitutions we treat in this work, the corresponding substitution operator is never compact; see \cite[Prop.~6.2]{MRW2}. A weaker condition is \emph{quasi-compactness}. An operator $M$ with spectral radius $1$ is called quasi-compact, if there exists a compact operator $N$ and $n\in\mathbb{N}$ such that $|M^n-N|^{ }_{\text{op}}<1$, where $|\cdot|^{ }_{\text{op}}$ is the operator norm. We refer to \cite{HH,MRW2} and references therein for details on such operators and their properties.

Most asymptotics for abelian quantities
(i.e., quantities which solely depend on the number of certain tiles within a supertile and not their location) for substitutions on finite alphabets are encoded by the corresponding substitution matrix. In the compact alphabet setting, one can associate a \emph{substitution operator} $M$ to the substitution $\varrho^{ }_{\bs{a}}$ as follows. Consider the Banach space $E=C(\mathcal{A})$ of continuous functions on $\mathcal{A}$ (with the sup norm). Define $M\colon E\to E$ to be 
\[
(Mf)(b)=\sum_{c\in \varrho^{ }_{\bs{a}}(b)} f(c),
\] 
where $\varrho^{ }_{\bs{a}}(b)$ is to be understood as a multiset. 
$M$ is a positive and bounded linear operator on $E$. 

Let $K$ be the positive cone in $E$ consisting of all non-negative continuous functions. A natural length function $\ell$ for $\varrho^{ }_{\bs{a}}$ is a function in $K$ for which $M\ell=\lambda\ell$ for some $\lambda\neq 0$. In other words, it corresponds to a non-negative eigenvector of $M$ in $K$ with non-zero eigenvalue $\lambda$.
If further $\lambda>1$ and $\ell$ is \emph{strictly positive}, i.e., $\ell(b)>0$ for all $b\in\mathcal{A}$, one can associate a geometric inflation rule to $\varrho^{ }_{\bs{a}}$ which generates substitution tilings on $\mathbb{R}$ with (possibly) infinitely many prototile lengths. The following result is proved in \cite{MRW2}. 

\begin{thm}\label{thm:length-function}
Let $\varrho$ be a substitution on a compact alphabet. Suppose that $\varrho$ is primitive and that the associated scaled substitution operator $\frac{1}{r(M)} M$ is quasi-compact. Then, $\varrho$ admits a unique (up to scalar multiplication) natural length function which is strictly positive with $\lambda=r(M)>1$.
\end{thm}

\begin{rem}\label{rem:quasi-compact}
Quasi-compactness has strong implications to the spectral properties of $M$. In particular, this is equivalent to the \emph{essential spectral radius} $r_{\text{ess}}(M)$ being strictly less than $r(M)$. Since it is not central to our arguments, we do not define $r_{\text{ess}}(M)$ here and refer the reader to \cite{Aksoy,MRW2,Pol} instead. This observation also implies that outside the essential spectral radius, there are at most finitely many other elements of $\sigma(M)$, each being an eigenvalue with finite-dimensional (generalised) eigenspace; see \cite[Ch.~XIV.1]{HH}. 
\end{rem}

It was shown in \cite[Sect.~4]{FGM} that the substitutions considered in Section~\ref{sec:def} satisfy the conditions in Theorem~\ref{thm:length-function}.
Moreover, we know closed forms for $\lambda$ and $\ell$ in terms of the defining sequence~$\bs{a}$. 

\begin{prop}[\cite{FGM}]  \label{prop:mulambda}
Let $\varrho^{ }_{\bs{a}}$ be a substitution on $\mathcal{A}=\mathbb{N}_0\cup \left\{\infty\right\}$ as defined in Eq.~\eqref{eq:pre-subs}. Let $\mu$ be the unique real number in $(0,1)$ that satisfies 
\[
\frac{1}{\mu}=\sum_{i=0}^{\infty}a_i\mu^i. 
\]
One then has 
\begin{equation}\label{eq:closed-form}
\lambda=\mu+\frac{1}{\mu}\quad \text{ and } \quad
\ell([k])=\mu^k+\sum_{j=1}^k\sum_{i=j}^\infty a_i\mu^{i+k+1-2j}, \text{ for } k>0,
\end{equation}
with the normalization $\ell([0])=1$.
\end{prop}

It is easy to see that the substitution defined in Eq.~\eqref{eq:pre-subs} admits a bi-infinite tiling fixed point $\mathcal{T}$ with seed $\infty|0$. From this substitution tiling, one can derive a  
Delone set $\varLambda_{\mathcal{T}}$ by collapsing each tile to the location of its left endpoint. Unique ergodicity implies that the points in $\varLambda_{\mathcal{T}}$ admit a well-defined frequency. This leads us to the following result; compare \cite[Prop.~4.5]{FGM}. 

\begin{prop}\label{prop:freqdens}
Let $\varLambda_{\mathcal{T}}$ be the Delone set derived from the tiling fixed point of $\varrho^{ }_{\bs{a}}$. Then, the frequency of the point of type $[k]$ in $\varLambda_{\mathcal{T}}$ is given by $\nu([k])=(1-\mu)\mu^k$. Moreover, the density of $\varLambda_{\mathcal{T}}$ exists and is given by\ $\dens(\varLambda_{\mathcal{T}})=\big(\sum_{k=0}^{\infty}\nu([k])\ell([k])\big)^{-1}$.
\end{prop}

\begin{figure}
    \centering
    \includegraphics[width=.7\textwidth]{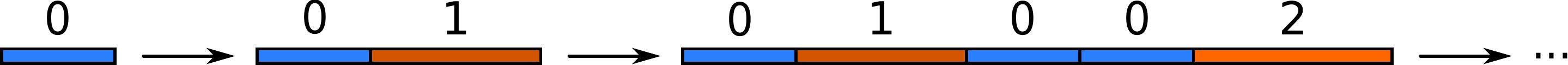}
    \caption{The $\varrho^{ }_{\bs{a}}$ of Example \ref{ex:all-1} applied to the tile $[0]$
    (a unit interval): $[0]$ is inflated by $\lambda=\frac{5}{2}$ and subdivided into tiles
    $[0]$ and $[1]$. This is the first order supertile $\varrho^{ }_{\bs{a}}([0])$. In
    the next step $\varrho^{ }_{\bs{a}}$ is applied to the two tiles of the first order
    supertile. This yields the second order supertile $\varrho^2_{\bs{a}}([0])$ consisting
    of five tiles.}
    \label{fig:bsp1111}
\end{figure}

\begin{exam}\label{ex:all-1}
Consider the sequence $\bs{a}$ with $a_i=1$ for all $i$. The corresponding substitution $\varrho^{ }_{\bs{a}}$ has $\mu=\sfrac 12$ and hence $\lambda=\sfrac 52$, $\ell([k])=2-\sfrac{1}{2^k}$ and $\nu([k])=\sfrac{1}{2^{k+1}}$;
compare Figure \ref{fig:bsp1111}. The average distance between two points in $\varLambda_{\mathcal{T}}$ is given by $(\dens(\varLambda_{\mathcal{T}}))^{-1}=\sfrac43$.  
\end{exam}

\subsection{Spectral gap, discrepancies, and bounded distance equivalence}

Apart from the existence of a strictly positive length function and unique ergodicity,
quasi-compactness (together with primitivity) has direct implications to studying discrepancies. 
It also implies that the operator $M$ has a \emph{spectral gap}, i.e., it admits a ``second largest'' element (in terms of modulus). More formally, the quantity 
\[
r_2=:\sup\left\{|\lambda^{\prime}|\colon \lambda^{\prime}\in \sigma(M)\setminus\left\{\lambda\right\}\right\}
\]
exists and is strictly less than the inflation factor $\lambda$; see \cite[Sec.~7]{MRW}. 

\begin{rem}
Not all primitive substitutions over a compact alphabet admit a spectral gap; see \cite[Ex.~6.18]{MRW} for an example for which $M$ is not quasi-compact.
\end{rem}

Let $f\in C(\mathcal{A})$ with $\|f\|\leqslant 1$. We define $\text{Act}(f,b,n):=M^{n}f(b)=\sum _{c\in\varrho^{n}(b)} f(c)$  and $\text{Exp}(f,b,n)=\lambda^n\cdot \ell(b) \dens(\varLambda)$.
(Later we will consider $f(b)=1$ for all $b \in \mathcal{A}$. Then  
$\text{Act}(f,b,n)$ is the actual number of tiles in $\varrho^n_{\bs{a}}(b)$, and $\text{Exp}(f,b,n)$ is the expected number of tiles in $\varrho^n_{\bs{a}}(b)$.)
The following discrepancy estimate follows from \cite[Thm.~7.3]{MRW}. 

\begin{thm}\label{thm:discr-gen}
Let $\varrho$ be a primitive substitution on a compact alphabet with quasi-compact substitution operator $M$. Then for any $f\in C(\mathcal{A})$ with $\|f\|\leqslant 1$, there exists a function $\theta\colon \mathbb{N}\to \mathbb{R}_{+}$ with $\lim_{n\to\infty} \sqrt[n]{\theta(n)}=1$ such that 
\begin{equation*}
|\textnormal{Exp}(f,b,n)-\textnormal{Act}(f,b,n)|\leqslant \theta(n)(r_2)^n. 
\end{equation*}
\end{thm}

\begin{rem}
Quasi-compact operators are ubiquitous in the study of dynamical systems (both in the discrete and continuous regimes). In particular, the quasi-compactness of the Ruelle--Perron--Frobenius transfer operator allows one to estimate rates of mixing and decay of correlations; see the seminal works by Ruelle and Pollicott \cite{Pol,Ruelle}. The (isolated) eigenvalues lying strictly between the peripheral spectrum (i.e., the intersection of $\sigma(M)$ with the circle of radius $r(M)$) and the essential spectrum figure in expansions similar to that in Theorem~\ref{thm:discr-gen}. These eigenvalues are also known as \emph{Ruelle--Pollicott resonances}, which find a variety of applications including fractal geometry and stochastic differential equations; see \cite{KK} for instance. In our setting, the peripheral spectrum consists solely of $\lambda$, and the other isolated eigenvalues of Ruelle--Pollicott-type are the ones which show up in our discrepancy estimates; see Example~\ref{ex:extra-eigen} below. 
\end{rem}

In this work, we will focus on the tile counting function $f$ with $f(b)=1$ for all $b\in\mathcal{A}$. Note that $(M^nf)(b)$ counts the number of tiles in $\varrho_{\bs{a}}^{n}(b)$. We will restrict ourselves to $n$-th order supertiles of type $[0]$, hence we fix $b=[0]$. For eventually constant sequences, we give certain bounds for $r_2$ and the function $\theta(n)$, which are related to Catalan numbers.

\section{Catalan numbers as discrepancies} \label{sec:catalan}

In this section, we completely describe the tile counting function and discrepancies for the substitution in Example~\ref{ex:all-1} where $a_i=1$ for all $i\geq 0$. The sequence serves as a model example since it is the first and probably one of the simplest examples where the discrepancy function shows Catalan-like growth.

For this sequence, recall from \cite{MRW} that we have the following pre-substitution on letters $[i],i\in \N$.

\begin{align*}
\varrho^{ }_{\bs a}([0]) & =[0][1], \text{ and}\\
\varrho^{ }_{\bs a}([i]) & =[0][i-1][i+1] \text{ for }i>0.
\end{align*}

To this pre-substitution we associate the corresponding infinite substitution matrix 
$${\bs A}=\left( 
\begin{array}{cccccc}
1 & 2 & 1 & 1 & 1 & \dots\\
1 & 0 & 1 & 0 & 0 & \ldots\\
0 & 1 & 0 & 1 & 0 & \ldots\\
0 & 0 & 1 & 0 & 1 & \ldots\\
0 & 0 & 0 & 1 & 0 & \ldots\\
\vdots & \vdots & \vdots & \vdots & \vdots & \ddots
\end{array}
\right).$$

One can view $\bs{A}$ as the transpose of the substitution operator $M$ restricted on the subspace associated to the isolated points in $\mathcal{A}$. 

We will estimate the {\it tile counting function} 
$$\#_{\bs a}(n):=(1,1,1,\ldots) \bs{A}^n (1,0,0,\ldots)^t=(M^{n}f)([0]),$$
which returns the total number of letters in $\varrho^n_{{\bs a}}([0])$. (Recall that we consider 
$\varrho^{ }_{\bs{a}}(b)$ as a multiset.) Note that despite using infinite vectors and an infinite matrix $\bs{A}$, the product is well defined because the product of the left vector with any power of $\bs{A}$ will give a vector with finitely many non-zero entries (since each column of $\bs{A}^n$ has only finitely many non-zero entries). In order to make formulas shorter we denote by $[\bs v]_i$ the $i$-th term of a vector $\bs v$. We are interested in $[(1,1,1,\ldots) \bs{A}^n]_0$ (we will start indexing from $0$, since the initial entry corresponds to $[0]$).

The coefficient of $\lambda^n$ in the formula for $[(1,1,1,\ldots) \bs{A}^n]_0$ can be estimated from the lengths of tiles and their frequencies, see \cite{MRW} or \cite{FGM}; compare also \cite{BG} for the finite alphabet case. In the particular case of this substitution, the average tile length is $\sfrac 43$ from Example~\ref{ex:all-1}. 

Since the inflation factor for this substitution is $\sfrac52$, the length of $\rho^n_{\bs{a}}([0])$ is $(\sfrac 52)^n$, and this iteration is expected to have $\sfrac 34\cdot (\sfrac 52)^n$ tiles. Thus we write 
$$[(1,1,1,\ldots) \bs{A}^n]_0=\frac 34 \cdot \left(\frac52\right)^n+d_{\bs a}(n)$$
and estimate the {\it discrepancy} term $d_{\bs a}(n)$ for all $n\geq 0$.

Let $\bs{I}$ be the infinite identity matrix. Then, by elementary calculation, $(-1,1,1,\ldots)=(1,1,1,\ldots)(2\bs{A}-5\bs{I})$. We define a new function $D(n)$ as
\begin{equation}\label{eq:D-function}
D(n):=[(-1,1,1,\ldots)\bs{A}^n]_0=2d_{\bs a}(n+1)-5d_{\bs a}(n).
\end{equation}

\begin{lem}\label{lem:catalan} 
Let $C_k$ be the $k$-th Catalan number, $C_k={2k \choose k}-{2k \choose k+1}=\frac{{2k \choose k}}{k+1}$. Then, 
$$D(n)=\left\{
\begin{array}{rl}
-C_k& \text{ if } n=2k,\\
0& \text{ if } n=2k+1.
\end{array}\right.
$$
\end{lem}
\begin{proof}
Let $\mathcal V$ be the space of all stabilising sequences ${\bs x}=(x_0,x_1,x_2,\ldots)$ of real numbers such that ${\displaystyle \sum_{i=0}^\infty\frac{x_i}{2^i}={\bs x} \cdot \left(1,\frac12,\frac14,\ldots\right)^t=0}$. The space $\mathcal V$ is invariant under the operator $\bs{B}$ defined by the right multiplication by $\bs{A}$. Indeed, $\bs{A} \left(1,\frac12,\frac14,\ldots\right)^t=\frac52\left(1,\frac12,\frac14,\ldots\right)^t$ and therefore if ${\bs x}\cdot \left(1,\frac12,\frac14,\ldots\right)^t=0$, then 
$$({\bs x}\bs{A})\cdot \left(1,\frac12,\frac14,\ldots\right)^t=\frac 52 \, {\bs x}\cdot \left(1,\frac12,\frac14,\ldots\right)^t=0.$$

Let ${\bs e}_0=(-1,1,1,\ldots)$ and for every $i>0$ we define ${\bs e}_i$ as follows
$$
\begin{array}{r}
{\bs e}_1=(1,-2,0,0,0,0,\ldots),\\
{\bs e}_2=(0,1,-2,0,0,0,\ldots),\\
{\bs e}_3=(0,0,1,-2,0,0,\ldots),
\end{array}$$
and so on. Then $\mathcal E=\{{\bs e}_i\}_i$ is a basis for $\mathcal V$: every vector in $\mathcal V$ is stabilising and it can be written as a multiple of ${\bs e}_0$ plus a finite linear combination of vectors from $\mathcal E$.

Simple computations show that $\bs{B}({\bs e}_0)={\bs e}_0+{\bs e}_1$ and  $\bs{B}({\bs e}_i)={\bs e}_{i-1}+{\bs e}_{i+1}$ for $i>0$. In other words, the matrix of the restriction of $\bs{B}$ on $\mathcal V$ in $\mathcal E$ is 
$${\bs B'}=\left( 
\begin{array}{cccccc}
1 & 1 & 0 & 0 & 0 & \dots\\
1 & 0 & 1 & 0 & 0 & \ldots\\
0 & 1 & 0 & 1 & 0 & \ldots\\
0 & 0 & 1 & 0 & 1 & \ldots\\
0 & 0 & 0 & 1 & 0 & \ldots\\
\vdots & \vdots & \vdots & \vdots & \vdots & \ddots
\end{array}
\right).$$

We claim that $\bs{B}^n({\bs e}_0)=c_{0,n}{\bs e}_0+ \cdots +c_{n,n}{\bs e}_n$, where the list $(c_{0,n}, \ldots,c_{n,n})$ is the list of binomial coefficients $\binom{n}{i}$, sorted from the largest one to the smallest one. That is, for $n$ even the list goes $\binom{n}{\sfrac{n}{2}}, \binom{n}{\sfrac{n}{2}-1}, \binom{n}{\sfrac{n}{2}+1}, \binom{n}{\sfrac{n}{2}-2}, \ldots, {n \choose 0}, {n \choose n}$; 
and for $n$ odd analogously.

This can be shown by in induction. The basis $n=0$ is trivial because ${\bs B}({\bs e}_0)={\bs e_0}=c_{0,0}{\bs e_0}$. For the induction step we write 
\begin{multline*}
    {\bs B}^{n+1}({\bs e_0})={\bs B}(c_{0,n}{\bs e}_0+ \cdots +c_{n,n}{\bs e}_n)=\\=(c_{0,n}+c_{1,n}){\bs e_0}+(c_{0,n}+c_{2,n}){\bs e_1}+(c_{1,n}+c_{3,n}){\bs e_2}+\cdots+(c_{n-2,n}+c_{n,n}){\bs e_{n-1}}+\\+c_{n-1,n}{\bs e_{n}}+c_{n,n}{\bs e_{n+1}}.
\end{multline*}

Note that the coefficients of ${\bs e}_n$ and ${\bs e_{n+1}}$ are $c_{n-1}=c_{n,n}=1=c_{n,n+1}=c_{n+1,n+1}$ as claimed. The remaining coefficients are sums of two binomial coefficients and we use the identity ${n \choose i}+{n\choose i+1}={n+1 \choose i+1}$ in each case. Note that $c_{i,n}$ is equal to $c_{i-1,n}$ or $c_{i+1,n}$ depending on the parities of $i$ and $n$ so the identity can be used for each coefficient.

Since only ${\bs e}_0$ and ${\bs e}_1$ have non-zero starting entry, 
$$[(-1,1,1,\ldots)\bs{A}^n]_0
=[\bs{B}^n({\bs e}_0)]_0=-c_{0,n}+c_{1,n}.$$
 The right-hand side is the difference between the second-to-the-largest and largest binomial coefficients which is either 0 if $n$ is odd, or $-C_k$ if $n$ is even.
\end{proof}

Since $d_{\bs a}(0)=\frac 14$ and $d_{\bs a}(n+1)=\frac52d_{\bs a}(n)+\frac12 D(n)$ from Eq.~\eqref{eq:D-function}, we derive the following explicit formula for $d_{\bs a}(n)$ for $n>0$:

$$
d_{\bs a}(n)=\left\{
\begin{array}{rl}
\frac12\left(\frac52\right)^{2k}\left(\frac 54-C_0-\left(\frac25\right)^2C_1-\ldots-\left(\frac25\right)^{2k}C_k\right),& \text{ if } n=2k+1,\\
\frac52d_{\bs a}(n-1),& \text{ if } n=2k.
\end{array}\right.
$$

Next, we use the power series for the Catalan numbers to obtain precise asymptotics for $d_{\bs{a}}(n)$. In particular,
$$C(x)=\frac{1-\sqrt{1-4x}}{2x}=\sum_{i=0}^\infty C_ix^i,$$
whose radius of convergence is $\sfrac14$.  Plugging $x=\frac{4}{25}$ into this series yields $C(\sfrac{4}{25})=\sfrac{5}{4}$, from which we get the representation of $d_{\bs a}(2k+1)$ as the tail
of this series:
$$d_{\bs a}(2k+1)=\frac12\left(\frac52\right)^{2k}\sum_{i=k+1}^\infty\left(\frac25\right)^{2i}C_{i}=\frac{2}{25}\left(C_{k+1}+\frac{4}{25}C_{k+2}+\left(\frac{4}{25}\right)^2C_{k+3}+ \cdots \right).$$

Since $C_{i+1}<4C_i$, the right-hand side is bounded from above by a geometric series with factor $\sfrac{16}{25}$ and we have the following estimates
$$\frac{2}{25}C_{k+1}\leq d_{\bs a}(2k+1)\leq \frac29 C_{k+1} \qquad \qquad \text{and} \qquad \qquad \frac{1}{5}C_{k+1}\leq d_{\bs a}(2k+2)\leq \frac59 C_{k+1}.$$

Overall, both $d_{\bs a}(2k-1)$ and $d_{\bs a}(2k)$ are in $\Theta(C_k)=\Theta\left(\frac{4^k}{k^{\sfrac32}}\right)$. Summarising, we get the 
following result.

\begin{thm}\label{thm:all1}
For the substitution $\rho_{\bs a}$ with $a_i=1$ for all $i$, the number
$\#_{\bs a}(n)$ of tiles in $\rho_{\bs a}^n([0])$ has asymptotic expansion
$$ \#_{\bs a}(n) \, \in \, \frac34\left(\frac52\right)^n+\Theta\left( \frac{2^n}{n^{\sfrac32}}\right).$$
In particular, $r_2=2$ and $\theta(n)=n^{-3/2}$ in Theorem~{\rm \ref{thm:discr-gen}}.
\end{thm}

\begin{rem}\label{rem:ess-spec}
Note that in the proof of quasi-compactness in \cite[Thm.~3.5]{FGM} the authors have implicitly provided a bound for $r_{\text{ess}}(M)$, namely $r_{\text{ess}}(M)\leqslant 2$ for all substitutions treated in this paper, which we conjecture to be an equality.  This suggests that, in the exact expansion of $\#_{\bs a}(n)$ of Theorem~\ref{thm:all1}, the error term completely comes from the essential spectrum. This is not true in general, as there are (at most finitely many) Ruelle--Pollicott-type eigenvalues which may show up; see Example~\ref{ex:extra-eigen}.
\end{rem}

\section{Stabilising sequences} \label{sec:stab}

In this section, we consider sequences ${\bs a}= (a_i)_i$ with $a_0>0$ that stabilise to a non-zero integer. So there is a $k$ such that $a_i=a>0$ for every $i\geq k$. We mostly employ the same approach as in Section \ref{sec:catalan}, but it requires more linear algebra to deal with technicalities. 

From Proposition \ref{prop:mulambda}, the inflation factor for the corresponding substitution
\begin{align*}
\varrho^{ }_{\bs a}([0])&=[0]^{a_0}[1], \text{ and}\\
\varrho^{ }_{\bs a}([i])&=[0]^{a_i}[i-1][i+1] \text{ for }i>0
\end{align*}
is equal to $\lambda=\mu+\dfrac1\mu$ where $\mu\in(0,1)$ is the solution of 
\begin{align}
\label{eq:mu}
\frac1\mu =\sum_{i=0}^\infty a_i\mu^i&=a_0+a_1\mu+ \cdots + a_{k-1}\mu^{k-1}+a\mu^k+a\mu^{k+1}+a\mu^{k+2}+ \cdots  \nonumber \\ 
& =a_0+a_1\mu+ \cdots + a_{k-1}\mu^{k-1}+\frac{a\mu^k}{1-\mu}.
\end{align}
From this, we can see that $\mu$ is an algebraic number, and so is $\lambda$. More specifically, $\mu$ is a root of the following polynomial with integer coefficients
$$P(x)=1+(-1-a_0)x+(a_0-a_1)x^2+ \cdots +(a_{k-2}-a_{k-1})x^k+(a_{k-1}-a)x^{k+1}.$$

As before, we define the substitution matrix as
$$\bs{A}=\left( 
\begin{array}{cccccc}
a_0 & a_1+1 & a_2 & a_3 & a_4 & \dots\\
1 & 0 & 1 & 0 & 0 & \ldots\\
0 & 1 & 0 & 1 & 0 & \ldots\\
0 & 0 & 1 & 0 & 1 & \ldots\\
0 & 0 & 0 & 1 & 0 & \ldots\\
\vdots & \vdots & \vdots & \vdots & \vdots & \ddots
\end{array}
\right)$$
and note that the elements in the first row stabilise to $a$ at some point. On the right, $\bs A$ can be multiplied by row-vectors from $\ell^1$, while on the left it can be multiplied by column-vectors from $\ell^{\infty}$. These spaces roughly correspond to tile lengths and frequencies in the finite alphabet setting.

We will estimate the number of letters in $\rho_{\bs a}^n([0])$, that is, $\#_{\bs a}(n) = [(1,1,1,\ldots)\bs{A}^n]_0$. The results of \cite{FGM} imply that the tiles have well defined frequencies $\nu([k])$, as well as the natural tile lengths $\ell([k])$, see Propositions \ref{prop:mulambda} and \ref{prop:freqdens}. As before, the leading term is $\sfrac{1}{c_{\bs a}}\lambda^n$. Since we are only interested in the discrepancy, we write
$$\#_{\bs a}(n):=[(1,1,1,\ldots)\bs{A}^n]_0=\frac{1}{c_{\bs a}}\lambda^n+d_{\bs a}(n),$$
where $c_{\bs a}$ is the expected tile length assuming $\ell([0])=1$. Note that $\sfrac{1}{c_{\bs{a}}}$ is exactly the density given in Proposition~\ref{prop:freqdens}.

Let $Q(x)=b_0+b_1x+ \cdots +b_mx^m$ be the minimal polynomial of $\lambda$. Define ${\bs x}_{\bs a}:=(1,1,1,\ldots)Q(\bs{A})$. For this vector, one has
\begin{equation}
\label{eq:conv}
[{\bs x}_{\bs a} \bs{A}^n]_0=[(1,1,1,\ldots)Q(\bs{A})\bs{A}^n]_0=b_0d_{\bs a}(n)+b_1d_{\bs a}(n+1)+ \cdots +b_md_{\bs a}(n+m),\end{equation}
as the exponential term $\lambda^n$ vanishes. For example, in Section \ref{sec:catalan},
we had $Q(x)=2x-5$ and ${\bs x}_{\bs a}=(-1,1,1,\ldots)$.

Below, we determine the growth rate of the left-hand side in order to estimate the growth rate of the right-hand side. In order to justify that, we need to establish that the right-hand side is non-zero.

\begin{lem}\label{lem:nonzero}
Suppose $\rho^{ }_{\bs a}$ is a non-constant length substitution. Then, for every non-zero polynomial $f(x)$ with integer coefficients, one has $(1,1,1,\ldots)f(\bs{A})\neq 0$.
\end{lem}
\begin{proof}
Let $x_i$ be the largest index such that in $(1,1,1,\ldots)\bs{A}^i$, the entries on places $x_i$ and $x_i+1$ are different. If all entries are equal, we set $x_i=-1$. Since $(1,1,1,\ldots)\bs{A}^i$ stabilises, the values $x_i$ are well defined for every $i\geq 0$ and $x_0=-1$.

We claim that the sequence $x_i$ is increasing. Indeed, $x_0=-1$ and $(1,1,1,\ldots)\bs{A}$ is not a vector with all equal entries because $\rho^{ }_{\bs a}$ is not a constant-length substitution. We consider several cases. Recall that $k$ is the index when sequence $\bs  a$ stabilises to the value~$a$.

{\bf Case 1:} $k>1$. Then $(1,1,1,\ldots)\bs{A}=(a_0+1,a_1+2,\ldots,a_{k-1}+2,a+2,a+2,\ldots)$ and $x_1=k-1$. Using induction, let $(1,1,1,\ldots)\bs{A}^i=(c_0,\ldots,c_{x_i},c,c,c,\ldots)$ with $x_i\geq k-1$ with $c_{x_i}\neq c$. Then, $[(1,1,1,\ldots)\bs{A}^{i+1}]_{x_i+1}=c_0a+c_{x_i}+c$, while $[(1,1,1,\ldots)\bs{A}^{i+1}]_{x_i+2}=c_0a+2c$ and these two entries are different. Therefore, $x_{i+1}>x_i$ which settles the case.

Here, we have used $x_i\geq k-1$. Thus $x_i+1\geq k$ and $c_0$ is multiplied by $a$ as the initial entry of the corresponding column of the matrix $\bs{A}$.

{\bf Case 2:} $k=0$. Then $a_i=a\geq 1$ and $(1,1,1,\ldots)\bs{A}=(a+1,a+2,a+2,\ldots)$ so $x_1=0>x_0$. After that, 
$$(1,1,1,\ldots)\bs{A}^2=(a^2+2a+2,a^2+3a+3,a^2+3a+4,a^2+3a+4,a^2+3a+4,\ldots)$$ and $x_2=1>x_1$. Now we can proceed as in the previous case.

{\bf Case 3:} $k=1$. Then $(a_i)_i=(b,a,a,a,\ldots)$ and $b\neq a+1$ since otherwise $\varrho^{}_{{\bs a}}$ would be a constant-length substitution. In this case, $(1,1,1,\ldots)\bs{A}=(b+1,a+2,a+2,\ldots)$ and $x_1=0>x_0$.  After that, 
$$(1,1,1,\ldots)\bs{A}^2=(b^2+a+b+2,ab+2a+b+3,ab+3a+4,ab+3a+4,ab+3a+4,\ldots)$$ and $x_2=1>x_1$. Again, the rest follows similarly to the first case.

Since $x_i$ is increasing, the vectors $(1,1,1,\ldots)\bs{A}^i$ are linearly independent, and this completes the proof.
\end{proof}

Let $\mathcal V$ be the space of all stabilising sequences ${\bs x}=(x_0,x_1,x_2,\ldots)$ of real numbers such that ${\displaystyle \sum_{i=0}^\infty x_i\mu^i=0}$. The space $\mathcal V$ is an invariant subspace of the operator $\bs{B}$ defined by the right multiplication by $\bs{A}$. Moreover, since $\bs{A} (1,\mu,\mu^2,\mu^3,\ldots)^t=\lambda (1,\mu,\mu^2,\mu^3,\ldots)^t$, we get that 
$${\bs x}_{\bs a} \cdot (1,\mu,\mu^2,\mu^3,\ldots)^t=(1,1,1,\ldots)Q(\bs{A})(1,\mu,\mu^2,\mu^3,\ldots)^t=0$$
and so ${\bs x}_{\bs a}\in\mathcal V$.

While $\mathcal V$ has a countable basis, it should be related to the minimal polynomial for $\mu$, which can be different from $P(x)$. Thus, we will work with a certain subspace of $\mathcal V$. In what follows, we define ${\bs e}_0=(-1,a_0,a_1,\ldots,a_{k-1},a,a,a,\ldots)$, and
\begin{equation}\label{eq:E-basis}
\begin{array}{r}
{\bs e}_1=(1,-1-a_0,a_0-a_1,a_1-a_2,\ldots, a_{k-2}-a_{k-1},a_{k-1}-a,0,0,0,0,\ldots),\\
{\bs e}_2=(0,1,-1-a_0,a_0-a_1,a_1-a_2,\ldots, a_{k-2}-a_{k-1},a_{k-1}-a,0,0,0,\ldots),\\
{\bs e}_3=(0,0,1,-1-a_0,a_0-a_1,a_1-a_2,\ldots, a_{k-2}-a_{k-1},a_{k-1}-a,0,0,\ldots),
\end{array}
\end{equation}
and so on. More specifically, the entries of ${\bs e}_0$ are taken from Eq.~\eqref{eq:mu} defining $\mu$, and for all other ${\bs e}_i$, the entries are the coefficients of the polynomial $P(x)$ originating from Eq.~\eqref{eq:mu} that has $\mu$ as a root. This implies that each ${\bs e}_i$ is in $\mathcal V$. Moreover, the span $\langle \mathcal E\rangle$ of  $\mathcal E:=\mathcal \{{\bs e}_i\}_{i\in \N_0}$ is a subspace of finite codimension in $\mathcal V$.

Again, straightforward computations show that $\bs{B}({\bs e}_0)={\bs e}_0+{\bs e}_1$ and  $\bs{B}({\bs e}_i)={\bs e}_{i-1}+{\bs e}_{i+1}$ for $i>0$. In other words, $\langle \mathcal E\rangle$ is an invariant subspace for $\bs{B}$ and  the matrix of the restriction of $\bs{B}$ on $\langle \mathcal E\rangle$ in $\mathcal E$ is 
$${\bs B'}=\left( 
\begin{array}{cccccc}
1 & 1 & 0 & 0 & 0 & \dots\\
1 & 0 & 1 & 0 & 0 & \ldots\\
0 & 1 & 0 & 1 & 0 & \ldots\\
0 & 0 & 1 & 0 & 1 & \ldots\\
0 & 0 & 0 & 1 & 0 & \ldots\\
\vdots & \vdots & \vdots & \vdots & \vdots & \ddots
\end{array}
\right).$$

\begin{lem}\label{lem:subspace}
There exists a non-zero polynomial $R(x)$ with integer coefficients such that ${\bs x}_{\bs a}R(\bs{A})$ is a (finite) linear combination of vectors from $\mathcal E$. 
\end{lem}
\begin{proof}
The vector ${\bs x}_{\bs a}$ is in $\mathcal V$ and $\mathcal V$ is an invariant subspace of $\bs{B}$. Since $\langle \mathcal E\rangle$ has a finite codimension in $\mathcal V$, the (cosets of) vectors ${\bs x}_{\bs a}, \bs{B}({\bs x}_{\bs a}), \bs{B}^2({\bs x}_{\bs a}), \ldots$ are linearly dependent in the factor space $\mathcal V/\langle \mathcal E\rangle$. Since all entries are integers, there exists a non-trivial integer linear combination of these vectors that lies in $\langle \mathcal E\rangle$.
\end{proof}

It is worth noting that Section \ref{sec:catalan} gives the simplest example and $R(x)=1$ because $\langle \mathcal E\rangle=\mathcal V$ in that section and ${\bs x}_\mathbf a\in \mathcal V$ right away. 

\begin{lem}\label{lem:poly}
For the polynomial $R(x)$ from Lemma \ref{lem:subspace}, there exists a non-zero polynomial $g(x)$ with integer coefficients such that ${\bs x}_{\bs a}R(\bs{A})=g(\bs{B})({\bs e}_0)$.
\end{lem}
\begin{proof}
The proof of existence follows from the fact that, for each $i$,  $\bs{B}^i({\bs e}_0)-{\bs e}_i$ belongs to the $\Z$-span of $\{{\bs e}_0,\ldots,{\bs e}_{i-1}\}$.
By Lemma \ref{lem:nonzero}, the left-hand side is non-zero and therefore $g(x)$ is nonzero as well.
\end{proof}

\begin{definition}
For a polynomial $\alpha(x)=\alpha_0+\alpha_1x+ \cdots +\alpha_m x^m\in \R[x]$ we define the {\it $\alpha$-twist} of  $d_{\bs a}$ as
$$\alpha*d_{\bs a}(n):=\alpha_0 d_{\bs a}(n)+\alpha_1d_{\bs a}(n+1)+ \cdots +\alpha_md_{\bs a}(n+m).$$

With this convention, the right-hand side in Eq.~\eqref{eq:conv} for $[{\bs x}_{\bs a} \bs{A}^n]_0$ can be written as $Q*d_{\bs a}$.
\end{definition}

The following lemma is straightforward.

\begin{lem}\label{lem:conv}
For every polynomial $\alpha(x)$, if $\bs y={\bs x}_{\bs a} \alpha(\bs{A})$, then $[\bs y \bs{A}^n]_0=(\alpha Q)*d_{\bs a}$.
\end{lem}

\begin{cor}\label{cor:twist}
In the notations used above, 
$$[{\bs x}_{\bs a}R(\bs{A})\bs{A}^n]_0=(RQ)*d_{\bs a}.$$
\end{cor}

Now we are ready to formulate an analogue of Lemma \ref{lem:catalan} for the general case of stabilising sequences.

\begin{lem}\label{lem:catlinear}
There exists an integer $p\geq 0$ and integers $\alpha_i,\beta_i$ such that
$$[{\bs x}_{\bs a}R(\bs{A})\bs{A}^n]_0=\left\{
\begin{array}{rl}
\alpha_0C_k+ \cdots  +\alpha_pC_{k+p}& \text{ if } n=2k,\\
\beta_0C_k+ \cdots  +\beta_pC_{k+p}& \text{ if } n=2k+1.
\end{array}\right.
$$
Furthermore, if $g(x)=\gamma_0+\gamma_1x+\cdots+\gamma_mx^m$ is the polynomial from Lemma \ref{lem:poly}, then $\alpha_i=-\gamma_{2i}$, and $\beta_{i+1}=-\gamma_{2i+1}$ for $i\geq 0$ while $\beta_0=0$. Thus, not all $\alpha_i, \beta_i$ are zero.
\end{lem}
\begin{proof}
From Lemma~\ref{lem:poly}, ${\bs x}_{\bs a}R(\bs{A})\bs{A}^n=\bs{B}^n(g(\bs{B})({\bs e}_0)).$ Suppose $g(x)=\gamma_0+\gamma_1x+\cdots+\gamma_mx^m$, then 
$${\bs x}_{\bs a}R(\bs{A})\bs{A}^n=\bs{B}^n(g(\bs{B})({\bs e}_0))=\gamma_0\bs{B}^n({\bs e}_0))+\gamma_1\bs{B}^{n+1}({\bs e}_0))+\cdots+\gamma_m\bs{B}^{n+m}({\bs e}_0)).$$

Since only ${\bs e}_0$ and ${\bs e}_1$ have non-zero first entry, we can use the approach of Lemma \ref{lem:catalan} to show that the first entry of $\bs{B}^i({\bs e}_0)$ is either $0$ or the negative of a Catalan number depending on the parity of $i$. 

More precisely, this approach implies that if $n=2k$, then
$$\alpha_0=-\gamma_0, \qquad \alpha_1=-\gamma_2, \qquad \alpha_2=-\gamma_4,$$
and so on, and if $n=2k+1$, then $\beta_0=0$ and
$$\beta_1=-\gamma_1, \qquad \beta_2=-\gamma_3, \qquad \beta_3=-\gamma_5,$$
and so on. These equalities give the claimed values for $\alpha_i$ and $\beta_i$ and also immediately show that not all $\alpha_i$ and $\beta_i$ are zero since $g(x)$ is a non-zero polynomial.
\end{proof}

\begin{prop}\label{lem:growth}
Let $p\geq 0$ and let $F(k)=\alpha_0C_k+ \cdots  +\alpha_pC_{k+p}$ for some integers $\alpha_i$ at least one of which is not $0$. Then, there exists an integer $0\leq q\leq p$ such that ${\displaystyle F(k)\in \Theta\left( \frac{4^k}{k^{q+\sfrac 32}}\right)}$. 
\end{prop}
\begin{proof}
First of all, note that the coefficients of the power series $\sum_k F(k)x^{k+p}$ for large $k$ coincide with the coefficients of the series for 
$$\frac{1-\sqrt{1-4x}}{2x}(\alpha_0x^p+ \cdots  +\alpha_p).$$ Since $\sqrt{1-4x}$ is not a rational function, the values $F(k)$ cannot all be zero starting for some $k$.

Recall that $C_{k+1}=\frac {2(2k+1)}{k+2} C_{k}$. Therefore
$$C_{k+i}= C_k \cdot \frac{2^i(2k+1)\cdots(2k+i)}{(k+2)\cdots (k+i+1)} = C_k \frac{P_i(k)}{Q_i(k)}$$ 
for some polynomials $P_i,Q_i$ of degree $i$. This implies that 
$$F(k)=C_k\cdot \mathcal R(k),$$ where $\mathcal R(k)$ is a rational function with integer coefficients and the numerator and denominator of $\mathcal R$ are polynomials of degree at most $p$. The numerator of $\mathcal R$ cannot be 0 because $F(k)$ does not stabilise to $0$. Thus, $\mathcal R(k)\in\Theta(\sfrac{1}{k^q})$ for some integer $q$ between $0$ and $p$. Taking into account the asymptotics for $C_k$ completes the proof.
\end{proof}

\begin{rem}
A similar result without the specific power of $k$ in the denominator can be obtained using algebraic generating functions. We refer to \cite[Ch. 6]{Sta1} for more details.
\end{rem}

\begin{thm}\label{thm:stab}
There is an integer $q$ such that a subsequence of $d_{\bs a}(n)$ belongs to $\Omega \left( \dfrac{2^n}{n^{q+\sfrac32}}\right)$.
\end{thm}

\begin{proof}
Assume to the contrary that $d_{\bs a}(n)\in o\left( \frac{2^n}{n^{q+\sfrac32}}\right)$ for every integer $q$. 
Using Lemma \ref{lem:catlinear} and Proposition \ref{lem:growth} we get that either for odd $n$ or for even $n$, the subsequence $[{\bs x}_{\bs a}R(\bs{A})\bs{A}^n]_0 $ is in $\Theta \left( \frac{2^n}{n^{q+\sfrac32}}\right)$ for some non-negative integer $q$.

From Corollary \ref{cor:twist}, $[{\bs x}_{\bs a}R(\bs{A})\bs{A}^n]_0=(RQ)*d_{\bs a}$. By our assumption, for every fixed $i\geq 0$, $d_{\bs a}(n+i)$ is in $o\left( \frac{2^n}{n^{q+\sfrac32}}\right)$ and so is the $RQ$-twist $(RQ)*d_{\bs a}$, which contradicts the previous estimate. 
\end{proof}

It is clear that $\Omega \left( \dfrac{2^n}{n^{q^{ }_0+\sfrac32}}\right)\subset \Omega \left( \dfrac{2^n}{n^{q+\sfrac32}}\right)$, 
for all $q\geqslant q_0$, and so one can assume without loss of generality that $q$ is non-negative in the theorem above. This formulation is closer to \emph{lower bounds} for asymptotic behaviours exhibited by substitutions on finite alphabets. We emphasise, however, that examples for which $d_{\bs{a}}(n)$ exhibits \emph{precise} asymptotics (as in Theorem~\ref{thm:all1}) with non-negative $q$ can only occur in the infinite alphabet setting. On the other hand, there also exist examples for which $q=-1$ in the actual asymptotics along a subsequence; see Example~\ref{ex:neg-q}.




\begin{rem} \label{rem:randomwalk} It is worth mentioning that there is an alternative way to prove Theorem \ref{thm:stab} without using Catalan numbers at all. First, one observes that the matrix $\frac{1}{2} B$ is the transition matrix of a random walk on the non-negative integers. The central limit theorem implies that this random walk tends to a Gaussian standard distribution for large $n$, and  can thus be described by $f(x)=\frac{1}{\sqrt{2 \pi n}}e^{\sfrac{-x^2}{2n}}$.

The $p$-th derivative $f^{(p)}$ can be related to the linear
combinations of Catalan numbers in Lemma \ref{lem:catlinear}.
It turns out that $f^{(k)}(0)=0$ for $k$ odd, and $f^{(k)}=c \cdot 
\frac{1}{n^{(2p+1)/2}}$ for $p$ even. Multiplying again with $2^n$ yields exactly the growth rate in Theorem \ref{thm:stab}. 

However, using the approach with Catalan numbers does not only yield Theorem \ref{thm:stab} but gives us exact values for the results in Sections \ref{sec:catalan} and \ref{sec:exam}. 
\end{rem}

\begin{cor} \label{cor:bde}
Suppose $\rho^{ }_{\bs a}$ is a non-constant length substitution. Then no Delone
set $\Lambda$ arising from $\rho^{ }_{\bs a}$ is bounded distance equivalent to 
$\alpha \Z$, for any $\alpha>0$. 
\end{cor}
This is just a consequence of Theorem \ref{thm:stab} together with an infinite
version of Hall's marriage theorem \cite{Rado}. 

\section{Two further examples --- precise values}\label{sec:exam}

In this section, we apply the general framework from the last section to two particular examples having the same values $\mu$ and $\lambda$. This yields exact expressions for the respective discrepancies and illustrates to what extent equal values for $\mu$ and $\lambda$ may still allow for different discrepancy functions; namely here: equal asymptotics, but with different constants.

\begin{exam}\label{ex:odd-zero}
For this example, we consider ${\bs a}=(1,2,2,2,\ldots)$. The associated substitution matrix is $$\bs{A}=\left( 
\begin{array}{cccccc}
1 & 3 & 2 & 2 & 2 & \dots\\
1 & 0 & 1 & 0 & 0 & \ldots\\
0 & 1 & 0 & 1 & 0 & \ldots\\
0 & 0 & 1 & 0 & 1 & \ldots\\
0 & 0 & 0 & 1 & 0 & \ldots\\
\vdots & \vdots & \vdots & \vdots & \vdots & \ddots
\end{array}
\right).$$

The equation that defines $\mu$ is
$$\frac1\mu=1+2\mu+2\mu^2+2\mu^3+ \cdots =-1+\frac{2}{1-\mu}.$$
It transforms into $1-2\mu-\mu^2=0$ so $\mu=\sqrt2-1$ and $\lambda=\mu+\sfrac{1}{\mu}=2\sqrt 2$, with the minimal polynomial for $\lambda$ being $Q(x)=x^2-8$.

Thus, ${\bs x}_{\bs a}=(1,1,1,\ldots)(\bs{A}^2-8\bs{I})=(-2,2,4,4,4,\ldots)$ and
$$[{\bs x}_{\bs a} \bs{A}^n]_0=(x^2-8)*d_{\bs a}=d_{\bs a}(n+2)-8d_{\bs a}(n).$$

Now we turn our attention to the space $\mathcal V$ of all stabilizing sequences that satisfy ${\displaystyle \sum x_i\mu^i=0}$ and to the subspace spanned by the vectors ${\bs e}_i$ where
$$
\begin{array}{l}
{\bs e}_0=(-1,1,2,2,2,2,\ldots),\\
{\bs e}_1=(1,-2,-1,0,0,0,\ldots),\\
{\bs e}_2=(0,1,-2,-1,0,0,\ldots),\\
{\bs e}_3=(0,0,1,-2,-1,0,\ldots),
\end{array}$$
and so on. These are the vectors from Eq.~\eqref{eq:E-basis}.

The vector ${\bs x}_a$ lies already in this subspace because ${\bs x}_{\bs a}=2{\bs e}_0$. So we can use $R(x)=1$ in Lemma \ref{lem:subspace}. Actually, in this particular example, $\mathcal V=\left\langle \mathcal{E}\right\rangle$, as it was in Section \ref{sec:catalan}. 

Also, in Lemma \ref{lem:poly}, we can use $g(x)=2$. Plugging this data into Corollary \ref{cor:twist} and Lemma \ref{lem:catlinear}, we get ${\bs x}_a\bs{A}^n=2\bs{B}^n({\bs e}_0)$ and therefore
$$d_{\bs a}(n+2)-8d_{\bs a}(n)=[{\bs x}_a\bs{A}^n]_0=\left\{
\begin{array}{rl}
-2C_k& \text{ if } n=2k,\\
0& \text{ if } n=2k+1.
\end{array}\right.
$$

At this point we can see that $d_{\bs a}(2k+1)=8^kd_{\bs a}(1)$ so this subsequence appears to be in $\Theta((2\sqrt 2)^n)$. Such a term must be a leading term in $\#_{\bs a} (n)=[(1,1,1,\ldots)\bs{A}^n]_0$ and should not appear in $d_{\bs a}(n)$. However, once we track the expected length $c_{\bs a}$, this will not be an issue.

More specifically, the frequency of the tile $[i]$ is $(1-\mu)\mu^i=(2-\sqrt 2)(\sqrt 2-1)^i$ and the length of the tile $[i]$ is $\sqrt 2 + 1 -\sqrt 2(\sqrt 2-1)^i$. 
This gives the expected tile length to be 
$$c_{\bs a}=\sum_{i=0}^\infty (2-\sqrt 2)(\sqrt 2-1)^i \Bigl( \sqrt 2 + 1 -\sqrt 2(\sqrt 2-1)^i \Bigr)=\sqrt 2.$$

One then has
$$\#_{\bs a} (n)=[(1,1,1,\ldots)\bs{A}^n]_0=\frac{1}{c_{\bs a}}\lambda^n+d_{\bs a}(n)=\frac{1}{\sqrt 2}(2\sqrt 2)^n+d_{\bs a}(n).$$
This gives $d_{\bs a}(0)=1-\sfrac{1}{\sqrt 2}$ and $d_{\bs a}(1)=0$ because $(1,1,1,...)\bs{A}=(2,4,4,4,\ldots)$. Therefore, $d_{\bs a}(2k+1)=0$ for every integer $k\geq 0$ due to the recurrence relation above.

If $n$ is even, a similar approach as in Section \ref{sec:catalan} using tails of the power series from Catalan numbers gives the growth rate $d_{\bs a}(n)\in\Theta\left( \frac{2^n}{n^{\sfrac32}}\right)$.
\end{exam}

The next example shares the same $\mu$ and $\lambda$ with Example~\ref{ex:odd-zero}, but one has $\mathcal{V}\neq \left\langle \mathcal{E} \right\rangle$. Hence, we will need to apply all steps for the general case in Section~\ref{sec:stab}.

\begin{exam} 
Consider ${\bs a}=(1,1,3,4,4,4,4,\ldots)$. The associated substitution matrix is $$\bs{A}=\left( 
\begin{array}{ccccccc}
1 & 2 & 3 & 4 & 4 & 4 & \dots\\
1 & 0 & 1 & 0 & 0 & 0 & \ldots\\
0 & 1 & 0 & 1 & 0 & 0 &\ldots\\
0 & 0 & 1 & 0 & 1 & 0 & \ldots\\
0 & 0 & 0 & 1 & 0 & 1 &\ldots\\
0 & 0 & 0 & 0 & 1 & 0 &\ldots\\
\vdots & \vdots & \vdots & \vdots & \vdots & \vdots & \ddots
\end{array}
\right).$$

The equation that defines $\mu$ is
$$\frac1\mu=1+\mu+3\mu^2+4\mu^3+4\mu^4+4\mu^5+ \cdots =-3-3\mu-\mu^2+\frac{4}{1-\mu}.$$
It transforms into $1-2\mu-2\mu^3-\mu^4=0$ or $(1-2\mu-\mu^2)(1+\mu^2)=0$ so $\mu=\sqrt2-1$ and $\lambda=\mu+\sfrac{1}{\mu}=2\sqrt 2$, with minimal polynomial $Q(x)=x^2-8$ as before.

However, ${\bs x}_{\bs a}=(1,1,1,\ldots)(\bs{A}^2-8\bs{I})=(-3,1,7,11,12,12,12,12,\ldots)$.

Since $\mu$ and $\lambda$ are the same as in the previous example, $\mathcal V$ is still the space of all stabilizing sequences that satisfy ${\displaystyle \sum x_i\mu^i=0}$. But the vectors ${\bs e}_i$ and the subspace they span are different now. In particular, $\mathcal E=\{{\bs e}_i\}_{i\geq 0}$ where
$$
\begin{array}{l}
{\bs e}_0=(-1,1,1,3,4,4,4,\ldots),\\
{\bs e}_1=(1,-2,0,-2,-1,0,0,0,\ldots),\\
{\bs e}_2=(0,1,-2,0,-2,-1,0,0,\ldots),\\
{\bs e}_3=(0,0,1,-2,0,-2,-1,0,\ldots),
\end{array}$$
and so on. In this case, the span $\langle \mathcal E\rangle$ is of codimension $2$ in $\mathcal V$ and ${\bs x}_{\bs a}$ is not in the span.

Observing that ${\bs x}_{\bs a}=3{\bs e}_0+(0,-2,4,2,0,0,0,\ldots)$, we write
$$
\begin{array}{rl}
{\bs x}_{\bs a}+\langle \mathcal E\rangle=&(0,-2,4,2,0,0,0,\ldots)+\langle \mathcal E\rangle, \text{ and}\\
{\bs x}_{\bs a}\bs{A} +\langle \mathcal E\rangle=&(-2,4,0,4,2,0,0,\ldots)+\langle \mathcal E\rangle=-2{\bs e}_1+\langle \mathcal E\rangle=\langle \mathcal E\rangle.\\
\end{array}$$

Thus, we can take $R(x)=x$ in Lemma \ref{lem:subspace} and write
$${\bs x}_{{\bs a}} \bs{A}=(3{\bs e}_0+(0,-2,4,2,0,0,0,\ldots))\bs{A}=3({\bs e}_0+{\bs e}_1)-2{\bs e}_1=3{\bs e}_0+{\bs e}_1.$$

For Lemma \ref{lem:poly}, we take $g(x)=x+2$ because $$(\bs{B}+2\bs{I})({\bs e}_0)=({\bs e}_0+{\bs e}_1)+2{\bs e}_0={\bs x}_{{\bs a}}\bs{A}={\bs x}_{{\bs a}} R(\bs{A}).$$

Then, ${\bs x}_{\bs a} R(\bs{A})\bs{A}^{n}=\bs{B}^{n}((\bs{B}+2\bs{I})({\bs e}_0))=\bs{B}^{n+1}({\bs e}_0) + 2\bs{B}^{n}({\bs e}_0)$. We use $RQ=x(x^2-8)$ in Corollary \ref{cor:twist} and Lemma \ref{lem:catlinear}  and write
$$d_{\bs a}(n+3)-8d_{\bs a}(n+1)=(x(x^2-8))*d_{\bs a}=[{\bs x}_{\bs a} R(\bs{A})\bs{A}^n]_0=\left\{
\begin{array}{rl}
-2C_k,& \text{ if } n=2k,\\
-C_{k+1},& \text{ if } n=2k+1.
\end{array}\right.
$$

Here, we can see that the two expressions for even and odd $n$ can be different linear combinations of (different) Catalan numbers. Nevertheless, we still have $d_{\bs a}(n)\in\Theta\left( \frac{2^n}{n^{\sfrac32}}\right)$ as the subsequences of odd and even $n$ can be written using tails of the Catalan series as in Section \ref{sec:catalan} and as in the previous example. However, the constants in the inequalities defining the order of growth are different for odd and even $n$.
\end{exam}

\section{Even more examples --- numerical evidence}\label{sec:moreexam}
In this section, we give more examples of eventually constant sequences ${\bs a}$ that lead to various behaviours of the tile counting function that have not appeared in this paper so far.  We provide examples of sequences which (i) give rise to eigenvalues different from $\lambda$ which show up in the asymptotics,  (ii) admit ``fake'' eigenvalues that appear in finding the general solution of the equation in Lemma \ref{lem:catlinear} treated as a linear recurrence relation, but do not appear in the actual numerical solution, and (iii)  illustrate the fact that, in Theorem \ref{thm:stab}, the parameter $q$ for actual asymptotics may have negative values as well.

For each example described below, our computational experiments in Wolfram Mathematica \cite{Mat} are organised as follows. 

Given the sequence ${\bs a}$, we generate the matrix $\bs{A}_1$ obtained from $\bs{A}$ by taking its upper-left corner of size $201\times 201$. For the vector ${\bs t}=(1,0,0,\ldots,0)^t$ of size 201 and for every $n=1,\ldots,200$, the sum of entries of $\bs{A}_1^n \bs t$ counts the number $\#_{\bs a}(n)$ of tiles in the supertile $\rho_{\bs a}^n([0])$ because none of these supertiles contain tiles outside of $\{[0],[1],\ldots,[200]\}$.

Using Proposition \ref{prop:mulambda}, we can find precise values of the respective $\mu$ and $\lambda$, and the density of the resulting Delone set. This allows us to find the discrepancy function $d_{\bs a} (n)$ for $n=1,\ldots,200$ and, in the cases below, infer its correct asymptotics.

\begin{rem} \label{rem:itlog}
Whereas the values in our computations are exact for $n$ up to
200, they are pretty accurate for much larger $n$: as outlined in Remark \ref{rem:randomwalk} we can describe the discrepancy as a random walk on the non-negative integers. By the law of the iterated logarithm \cite{khi,kol},
our random walk stays below position $n$ after essentially $n^2$ steps
almost surely (more precisely, after $\sqrt{2 \log \log n^2} n^2$ steps or less). So we might alternatively have carried out the computations up
to $n=10^4$ or so, without losing much precision. As before, we decided
to restrict ourselves to exact values.
\end{rem}

The frequencies and density can be obtained from Proposition \ref{prop:freqdens}. An alternative, and for most cases described below more useful, method to get the length function $\ell$ is to treat it as a left eigenvector of $\bs{A}$ with eigenvalue $\lambda$. The structure of $\bs{A}$ ensures that, starting from some $k$, $\ell$ satisfies the equation
$$a+\ell([k-1])+\ell([k+1])=\lambda \ell([k]),$$
where $a$ is the value to which ${\bs a}$ stabilises. This equation can be treated as non-homogeneous linear recurrence which can be solved explicitly knowing several starting terms (in this case, we have $k+1$ initial conditions). We can also simplify the process a bit by using that $\ell$ must be continuous, and hence bounded on the alphabet $\mathcal A$. 

For all examples described below, we have $k=1$ or $k=2$, so the initial values of the corresponding recurrence as well as the density can be computed manually. We illustrate how this approach works in the example below. 

\begin{exam}\label{ex:extra-eigen}
Consider the sequence ${\bs a}=(1,9,9,9,\ldots)$. As we will see, the bound from Theorem \ref{thm:stab} is a lower bound in this case, but the actual growth of the discrepancy is exponential with base greater than $2$. Nevertheless, we can guess the leading term of the discrepancy by finding additional eigenvectors of the infinite matrix $\bs{A}$, and by isolating the corresponding exponential term; the remaining part of discrepancy still exhibits Catalan-like growth.

Using Eq.~\eqref{eq:mu}, $\mu$ can be computed from the equation
$$\frac 1\mu=\sum a_i\mu^i=1+9(\mu+\mu^2+\mu^3+ \cdots )=1+\frac{9\mu}{1-\mu}.$$
This transforms into the quadratic equation $8\mu^2+2\mu-1=0$ with two solutions $\mu=\sfrac 14$ and $\mu_*=\sfrac{-1}{2}$, where the former defines the inflation factor $\lambda =\sfrac{17}{4}$.

In order to find the lengths, we recall that the vector
$(1,\ell([1]),\ell([2]),\ell([3]),\ldots)$ is a left $\lambda$-eigenvector of the matrix 
$$\bs{A}=\left( 
\begin{array}{cccccc}
1 & 10 & 9 & 9 & 9 & \dots\\
1 & 0 & 1 & 0 & 0 & \ldots\\
0 & 1 & 0 & 1 & 0 & \ldots\\
0 & 0 & 1 & 0 & 1 & \ldots\\
0 & 0 & 0 & 1 & 0 & \ldots\\
\vdots & \vdots & \vdots & \vdots & \vdots & \ddots
\end{array}
\right).$$
Therefore, $1+\ell([1])=\sfrac{17}{4}$ and $\ell([1])=\sfrac{13}{4}$. Additionally, for $k\geq 1$,
$$9+\ell([k-1])+\ell([k+1])=\frac{17}{4}\ell([k]).$$
Rewriting the above equation into
$$\ell([k+1])-\frac{17}{4}\ell([k])+\ell([k-1])=-9,$$
we can treat it as a non-homogeneous linear recurrence with solution 
$\ell([k])=\alpha\cdot \left(\sfrac14\right)^k+\beta \cdot 4^k+4$,
with initial terms $\ell([0])=1$ and $\ell([1])=\sfrac{13}{4}$.

The initial values allow us to find $\alpha =-3$ and $\beta =0$. Note that the latter guarantees that $\ell([\cdot])$ is bounded, as it is required for left eigenvectors of $\bs{A}$. Overall, $\ell([k])=4-\sfrac{3}{4^k}$.

Recalling that the frequency of the letter $[k]$ is $\nu([k])=(1-\mu)\mu^k$, we get the average tile length to be
$$c_{\bs{a}}=\sum_{k=0}^\infty \ell([k])\nu([k])=\sum_{k=0}^\infty \left(4-\frac{3}{4^k}\right)\cdot \left(1-\frac14\right)\frac{1}{4^k}=\frac85.$$
Hence the corresponding density is $\sfrac 58$. 

As before, the density gives the coefficient for the leading term. That is, 
$$\#_{\bs a}(n) \; \in \; \frac 58 \left (\frac{17}{4} \right)^n + o \left(\left (\frac{17}{4} \right)^n\right).$$
For the discrepancy, we want to estimate
$\#_{\bs a}(n)- \frac 58 \left (\frac{17}{4} \right)^n$.

Initally, we expected that this difference will grow as $\dfrac{2^n}{n^{\sfrac{3}{2}}}$ as in Theorem \ref{thm:all1}, but it turns out that it actually grows faster. The numerical computations suggest that the growth rate comes from the second solution $\mu_*=\sfrac{-1}{2}$ of the quadratic equation for $\mu$ discussed above.

More precisely, $\mu_*$ gives rise to $\lambda_*=\mu_*+\sfrac{1}{\mu_*}=\sfrac{-5}{2}$ which is an eigenvalue of $\bs{A}$ with left eigenvector $\ell_*$ given by
$$\ell_*([k])=\mu_*^k+\sum_{j=1}^k\sum_{i=j}^\infty a_i\mu_*^{i+k+1-2j};$$
compare with Eq.~\eqref{eq:closed-form}. For every $k\geq 0$, the corresponding series are convergent because ${\bs a}$ is bounded and $|\mu_*|<1$.

In the finite-dimensional case, every eigenvalue of an operator ${\bs A}$ may trigger corresponding exponential growth in the image of ${\bs A}^n$. A similar effect can be observed here. Since all of our operators here are quasi-compact, they all admit finitely many eigenvalues outside the essential spectrum, which means the expansion $\#_{\bs a}(n)$ contains finitely many (exponential) terms before the 
Catalan-like term; see Remarks~\ref{rem:quasi-compact} and \ref{rem:ess-spec}. 

In particular, our computations show that
$$\#_{\bs a}(n)- \frac 58 \left (\frac{17}{4} \right)^n \; \in \; \frac 14 \left (-\frac{5}{2} \right)^n+o\left(\left (\frac{5}{2} \right)^n\right)$$
and the difference shows Catalan-like growth rate of $\dfrac{2^n}{n^{\sfrac{3}{2}}}$, see Figure \ref{fig:secondeigenvalue}. 

\begin{figure}
    \centering
    \includegraphics[width=0.5\textwidth]{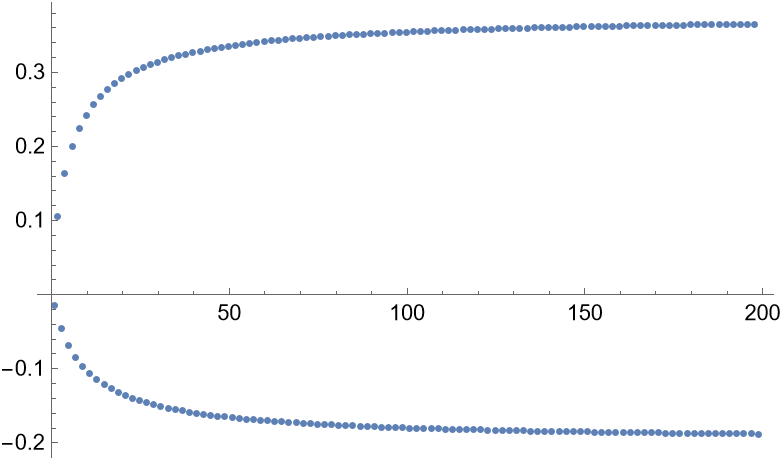}
    \caption{For the sequence ${\bs a}=(1,9,9,9,\ldots)$, the plot shows the ratio $$\frac{\#_{\bs a}(n)- \frac 58 \left (\frac{17}{4} \right)^n- \frac 14 \left (-\frac{5}{2} \right)^n}{\sfrac{2^n}{n^{\sfrac{3}{2}}}}.$$.}
    \label{fig:secondeigenvalue}
\end{figure}

Overall, we conjecture that, for ${\bs a}=(1,9,9,9,\ldots)$, we have
\begin{equation}\label{eq:secondeigenvalue}
\#_{\bs a}(n) \; \in \; \frac 58 \left (\frac{17}{4} \right)^n+ \frac 14 \left (-\frac{5}{2} \right)^n+\Theta\left(\frac{2^n}{n^{\sfrac32}}\right).
\end{equation}
\end{exam}

\begin{exam} For the next example, we consider the sequence ${\bs a}=(3,1,1,1,\ldots)$. This sequence was studied by Ma\~nibo, Rust, and Walton in \cite[Ex.~6.14]{MRW2}. This sequence gives rise to alternative values $\mu_*$ and $\lambda_*$ as well, but $\lambda_*$ is not an eigenvalue of $\bs{A}$ and does not appear in the counting function $\#_{\bs a}(n)$.

As before, we use  Eq.~\eqref{eq:mu} to write
$$\frac1\mu=3+\mu+\mu^2+\mu^3+ \cdots =3+\frac{\mu}{1-\mu}.$$
This transforms into the quadratic equation $2\mu^2-4\mu+1=0$ with two solutions $\mu=1-\sfrac{1}{\sqrt 2}$ and $\mu_*=1+\sfrac{1}{\sqrt 2}$. The former defines $\lambda =3+\sfrac{1}{\sqrt 2}$, the inflation factor also obtained in \cite[Ex.~6.14]{MRW2}.

Using a similar approach to the lengths, we find that the average length of tiles is $\frac{12-4\sqrt 2}{7}$. Thus,
$$\#_{\bs a}(n) \; \in \; \frac{7}{12-4\sqrt 2}\left(3+\frac{1}{\sqrt 2}\right)^n+o\left(\left(3+\frac{1}{\sqrt 2}\right)^n \right).$$
More precisely, our experiments show that, for the sequence ${\bs a}=(3,1,1,1,\ldots)$,
\begin{equation}\label{eq:nosecond}
\#_{\bs a}(n) \; \in \;  \frac{7}{12-4\sqrt 2}\left(3+\frac{1}{\sqrt 2}\right)^n+\Theta\left(\frac{2^n}{n^{\sfrac32}}\right),\end{equation}
see Figure \ref{fig:nosecond}.

\begin{figure}
    \centering
    \includegraphics[width=0.5\textwidth]{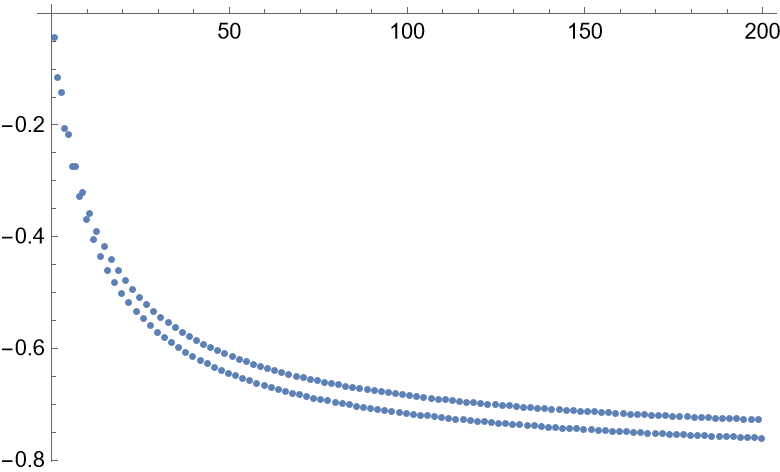}
    \caption{For the sequence ${\bs a}=(3,1,1,1,\ldots)$, the plot shows the ratio $$\frac{\#_{\bs a}(n)- \frac{7}{12-4\sqrt 2}\left(3+\frac{1}{\sqrt 2}\right)^n}{\sfrac{2^n}{n^{\sfrac{3}{2}}}}.$$.}
    \label{fig:nosecond}
\end{figure}

Note that $\lambda_*=\mu_*+\sfrac{1}{\mu_*}=3-\sfrac{1}{\sqrt 2}>2$ does not show up in the asymptotics of $\#_{\bs a}(n)$ because it is not an eigenvalue for $\bs{A}$. 
From Remark~\ref{rem:ess-spec}, if $\lambda_{\ast}$ is in the spectrum of the substitution operator, and $|\lambda_{\ast}|>2$, then $\lambda_{\ast}$ must be an eigenvalue, with eigenvector of the form given in Proposition~\ref{prop:mulambda}. However, if $|\mu_{\ast}|>1$, the corresponding eigenvector $\ell_{\ast}([k])$ is unbounded, which means $\lambda_{\ast}$ cannot be an eigenvalue. Since $\mu_*$ does not yield an eigenvector, we call $\lambda_*$ a {\it fake} eigenvalue. 

The fake eigenvalue $\lambda_*$ will appear in the general solution of the equation in Lemma~ \ref{lem:catlinear} treated as a linear recurrence. Indeed, the polynomial $R(\bs{A})$ there contains the minimal polynomial of $\lambda$ as a multiple. Since $\lambda$ and $\lambda_*$ are algebraically conjugate, they both will be roots of the corresponding characteristic polynomial. However, the initial conditions will eliminate the term $\lambda_*^n$ from the general solution, consistent with $\lambda_{\ast}$ not being in the spectrum of $M$. 

We also experimented with a slightly modified sequence ${\bs a}=(a_0,1,1,1,\ldots)$ with $a_0=4,5,6$ with the same results. For each $a_0$, there is a fake $\lambda_*$ which does not contribute to $\#_{\bs a}(n)$.

We briefly remark that both examples from Section \ref{sec:exam} exhibit similar properties, namely, both of them additionally have $\mu_*=-1-\sqrt 2$ and $\lambda_*=-2\sqrt 2$ (which has the same absolute value as $\lambda=2\sqrt 2$), but this $\lambda_*$ is again fake as it does not give rise to an eigenvector and does not appear in the corresponding tile counting function. 
\end{exam}

\begin{exam}\label{ex:neg-q}
Next, we consider ${\bs a}=(2,4,2,2,2,\ldots)$. The equation for $\mu$ is
$$\frac1\mu=2+4\mu+2\mu^2+2\mu^3+ \cdots =2\mu +\frac{2}{1-\mu}.$$
It transforms into the cubic equation $2\mu^3-2\mu^2-3\mu+1=0$ or $(\mu +1)(2\mu^2-4\mu+1)=0$. So in addition to $\mu=1-\sfrac{1}{\sqrt 2}$ and $\mu_*=1+\sfrac{1}{\sqrt2}$ from the previous example, we also have $\mu_{**}=-1$.

Thus, $\lambda=3+\sfrac{1}{\sqrt 2}$ is the same as in the previous example. The average length of tiles is $4-2\sqrt 2$. We note that, as before, $\lambda_*=3-\sfrac{1}{\sqrt 2}$ does not contribute to the counting function $\#_{\bs a}(n)$.

Additionally, $\lambda_{**}=\mu_{**}+\sfrac{1}{\mu_{**}}=-2$ is not an eigenvalue of $\bs{A}$, but it affects the behaviour of $\#_{{\bs a}}(n)$. In particular, our experiments show that it alters the power of $n$ in the denominator of the discrepancy and
\begin{equation}\label{eq:-1}
\#_{\bs a}(n)\in \frac{1}{4-2\sqrt 2}\left(3+\frac{1}{\sqrt 2}\right)^n+\Theta\left(\frac{2^n}{n^{\sfrac12}}\right),\end{equation}
see Figure \ref{fig:-1}.

Note that this is still consistent with Theorem~\ref{thm:discr-gen}, with $r_2=2$ but now $\theta(n)=n^{-1/2}$. We conjecture that it is possible to tie this behaviour to the recurrence from Lemma~\ref{lem:catlinear} and to the spectral nature of $\lambda$. The right-hand side there grows at least as $\sfrac{2^n}{n^{q+\sfrac32}}$ with some alternating behaviour for odd/even $n$, so it is possible that having $-2$ as a root of characteristic polynomial adds a linear factor to the growth of some particular solution. 

Moreover, $\lambda=-2$ is an \emph{approximate eigenvalue} of $\bs{A}$, i.e., there exists a sequence $\left\{\bs{v}^{(n)}\right\}$ of unit vectors in $\ell^1$ for which 
$|\bs{A}\bs{v}^{(n)}+2\bs{v}^{(n)}|\to 0$ as $n\to \infty$. In this case, one can choose $\bs{v}^{(n)}$ to be 
\[
\bs{v}^{(n)} =\frac{1}{n}\left((-1)^{n}n,(-1)^{n-1}(n-1),\ldots,-1,0,0,\ldots\right).
\]
It is easy to see that every $\bs{v}^{(n)}$ has unit sup-norm in $\ell^1$. Moreover, one can check that 
\[
\bs{A}\bs{v}^{(n)}+2\bs{v}^{(n)}=\begin{cases}
(\frac{3}{n},0,0,\ldots),& n\text{ even},\\
(-\frac{4}{n},0,0,\ldots),& n\text{ odd},
\end{cases}
\]
from which the required convergence follows. Since $-2$ is an approximate eigenvalue which is not an eigenvalue, it is part of the \emph{continuous spectrum} of $\bs{A}$. 
A more detailed analysis of these approximate eigenvectors might yield a better description of the asymptotics. 
\end{exam}

\begin{figure}[!h]
    \centering
    \includegraphics[width=0.5\textwidth]{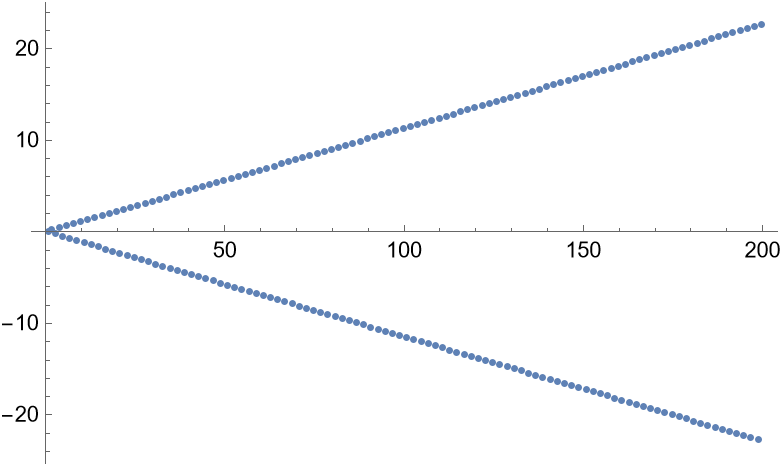}
    \caption{For the sequence ${\bs a}=(2,4,2,2,\ldots)$, the plot shows the ratio $$\frac{\#_{\bs a}(n)- \frac{1}{4-2\sqrt 2}\left(3+\frac{1}{\sqrt 2}\right)^n}{\sfrac{2^n}{n^{\sfrac{3}{2}}}}.$$ Unlike the previous plots this
    one indicates linear growth, hence an additional factor $n$.}
    \label{fig:-1}
\end{figure}

\begin{figure}[!h]
    \centering
    \includegraphics[width=0.5\textwidth]{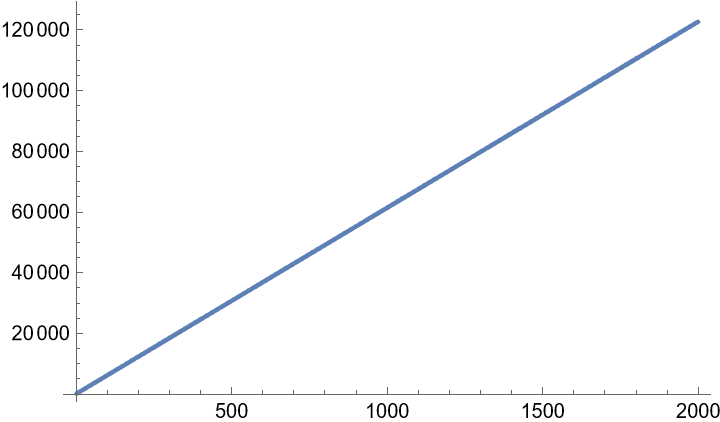}
    \caption{For the sequence ${\bs a}=(1,8,12,12,\ldots)$, the plot shows the expression $$\left(\frac{\#_{\bs a}(n)- \frac{3}{5}\left(\frac{17}{4}\right)^n-\frac15(-2)^n}{2^n}\right)^{-2}.$$}
    \label{fig:-1rep}
\end{figure}

\begin{exam}
Here, we consider ${\bs a}=(1,8,12,12,12,\ldots)$. The equation satisfied by $\mu$ reads $4\mu^3+7\mu^2+2\mu-1=(\mu+1)^2(4\mu-1)=0$. So it has $\mu=\sfrac14$ as the solution that gives $\lambda=\sfrac{17}{4}$ and $\mu_*=-1$ as root of multiplicity $2$.

However, the discrepancy between $\#_{\bs a}(n)$ and the leading term grows not as $2^n\sqrt n$ as one may guess from the previous example if every repetition of $\mu_*=-1$ (or $\lambda_*=-2$) brings a factor of $n$ to the discrepancy.

What is also different from previous examples is that $\lambda_*=-2$ is an eigenvalue with the eigenvector $(1,-3,-3,-3,\ldots)$. Moreover, this eigenvalue contributes to the tile counting function. 

Our computations (see Figure \ref{fig:-1rep}) suggest that
\begin{equation}\label{eq:-1rep}
\#_{\bs a}(n)\in \frac{3}{5}\left(\frac{17}{4}\right)^n+\frac15(-2)^n+\Theta\left(\frac{2^n}{n^{\sfrac12}}\right).\end{equation}

We note that the asymptotics of the tile counting function for this example is considerably more subtle than the rest, and we needed to check the 2000th iteration of the substitution instead of just the 200th.
\end{exam}

\begin{exam}
Our last example shows that complex values for $\mu$ may lead to relatively standard behaviour as well.

For the sequence ${\bs a}=(1,7,15,15,15,\ldots)$, the equation for $\mu$ has solutions $\mu=\sfrac14$ and $\mu_*=\frac{-1\pm i\sqrt 5 }{6}$. The inflation factor $\lambda=\sfrac{17}{4}$ defines the leading term of $\#_{{\bs a}}(n)$ and $\lambda_*=\frac{-7\pm i\cdot 5\sqrt 5}{6}$ are eigenvalues as well, since $|\mu_*|<1$ for both choices of the sign.

Again, treating the equation in Lemma \ref{lem:catlinear} as a linear recurrence, these $\lambda_*$'s should contribute an exponential-times-trigonometric function to the tile counting function $\#_{{\bs a}}(n)$; and this is exactly what our computations show, see Figure \ref{fig:complex}.
A refined expected formula for $\#_{{\bs a}}(n)$ with the exact trigonometric factor remains to be found.

\begin{figure}[!h]
    \centering
    \includegraphics[width=0.5\textwidth]{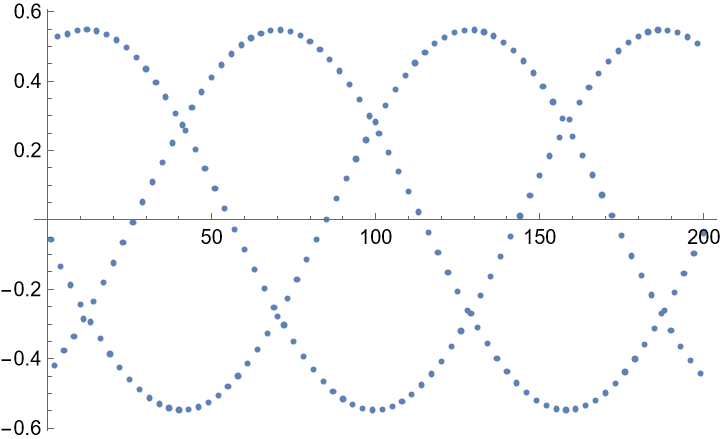}
    \caption{For the sequence ${\bs a}=(1,7,15,15,\ldots)$, the plot shows the ratio $$\frac{\#_{\bs a}(n)- \frac{1}{2}\left(\frac{17}{4}\right)^n}{\left(\sqrt{\frac{29}{6}}\right)^n}.$$}
    \label{fig:complex}
\end{figure}

\end{exam}

\section*{Acknowledgments}

The authors want to thank Igor Pak for discussions related to the asymptotics of Catalan-related sequences and corresponding references. We also would like to thank Lorenzo Sadun for numerous discussions on discrepancies of substitutions and Michael Baake for helpful comments on the manuscript. We also thank two anonymous reviewers for their valuable comments and suggestions; in particular, one for pointing the connection to one-dimensional random walks. Finally, we thank Fernando Cordero for discussions on random walks. A.G. is partially supported by the Alexander von Humboldt Foundation. N.M. is supported by the German Academic Exchange Service (DAAD) through a PRIME Fellowship.


\begin{thebibliography}{20}

\bibitem{Aksoy}
A.G.~Aksoy:
The radius of the essential spectrum, \textit{J.~Math.~Anal.~Appl.} 128 (1987) 101--107.

\bibitem{BG}
M.~Baake, U.~Grimm:
\textit{Aperiodic Order. Vol.~1: A Mathematical Invitation},
Cambridge University Press, Cambridge, 2013.

\bibitem{BGF}
M.~Baake, U.~Grimm, D.~Frettl\"oh, Pinwheel patterns and powder diffraction, \textit{Philos.~Mag.} 87 (2007) 2831--2838, \url{https://arxiv.org/abs/math-ph/0610012}.

\bibitem{FG}
D.~Frettlöh, A.~Garber: 
Pisot substitution sequences, one dimensional cut-and-project sets and bounded remainder sets with fractal boundary, \emph{Indag.~Math.} 29 (2018), 1114--1130, 
\url{https://arxiv.org/abs/1711.01498}.

\bibitem{FGM}
D.~Frettl\"oh, A.~Garber, N.~Ma\~nibo:
Substitution tilings with transcendental inflation factor, preprint (2022) 
\url{https://arxiv.org/abs/2208.01327}.

\bibitem{FGS}
D.~Frettl\"oh, A.~Garber, L.~Sadun:
Number of bounded distance equivalence classes in hulls of repetitive Delone sets. 
\emph{Discr.~Contin.~Dynam.~Syst.} 42 (2022) 1403–1414, \url{https://arxiv.org/abs/2101.02514}.

\bibitem{GD}
U.~Grimm,  X.~Deng, Some comments of pinwheel tilings and their diffraction, \textit{J.~Phys.~Conf.~Ser.} 284 (2011) 012032 (9pp), 
\url{https://arxiv.org/abs/1102.1750}.

\bibitem{HH}
H.~Hennion, L.~Herv\'{e}:\textit{Limit Theorems for Markov Chains and Stochastic Properties of Dynamical Systems by Quasi-compactness}, Springer Berlin (2001). 

\bibitem{HKK}
A.~Haynes, M.~Kelly, H.~Koivusalo:
Constructing bounded remainder sets and cut-and-project sets which are 
bounded distance to lattices II,
\emph{Indag.~Math.} 28 (2017) 138-144, \url{https://arxiv.org/abs/1602.00529}.

\bibitem{KK}
M.~Kesseb\"ohmer, S.~Kombrink: A complex Ruelle--Perron--Frobenius theorem for infinite Markov shifts with applications to renewal theory, \textit{Discr.~Contin.~Dynam.~Syst. S} 10 (2017) 335--352, \url{https://arxiv.org/abs/1604.08252}.

\bibitem{khi}
A.~Khinchin: 
\"Uber einen Satz der Wahrscheinlichkeitsrechnung, \emph{Fundam.~Math.} 6 (1924) 9–-20.

\bibitem{K}
D.E.~Knuth: 
\emph{The Art of Computer Programming, Vol.~1: Fundamental Algorithms}, 3rd ed., Addison Wesley, 1997.

\bibitem{kol}
A.~Kolmogoroff: 
\"Uber das Gesetz des iterierten Logarithmus, \emph{Mathematische Annalen} 101 (1929) 126-–135. 

\bibitem{MRW} 
N.~Ma\~nibo, D.~Rust,  J.J.~Walton: 
Spectral properties of substitutions on compact alphabets, preprint (2021), \url{https://arxiv.org/abs/2108.01762}.

\bibitem{MRW2} 
N.~Ma\~nibo, D.~Rust, J.J.~Walton: Substitutions on compact alphabets, preprint (2022), \url{https://arxiv.org/abs/2204.07516}.

\bibitem{oeis}
OEIS Foundation Inc.~(2022), The Catalan numbers, Entry A000108 in The On-Line Encyclopedia of Integer Sequences, \url{http://oeis.org/A000108}.

\bibitem{Pak}
I. Pak:
History of Catalan numbers. Appendix to R.~Stanley, ``Catalan Numbers'', Cambridge Univ. Press, 2015, 177-189, \url{https://arxiv.org/abs/1408.5711}.

\bibitem{Pol}
M.~Pollicott: Meromorphic extensions of generalized zeta functions, \textit{Invent.~Math.} 85 (1986) 147--164. 

\bibitem{Rado}
R.~Rado: Factorization of even graphs, \emph{Quart.~J.~Math.~Oxford} 20 (1949) 95--104.

\bibitem{Ruelle}
D.~Ruelle: Locating resonances for Axiom A dynamical systems, \textit{J.~Stat.~Phys.} 44 (1986) 281--292.

\bibitem{SS}
Y.~Smilansky, Y.~Solomon: 
A dichotomy for bounded displacement equivalence of Delone sets,
\emph{Ergodic Th.~Dynam.~Syst.} 42 (2022) 2693-–2710, \url{https://arxiv.org/abs/2011.00106}.

\bibitem{Sol14}
Y.~Solomon: A simple condition for bounded displacement, 
\emph{J.~Math.~Anal.~Appl.} 414 (2014) 134--148, \url{https://arxiv.org/abs/1111.1690}.


\bibitem{Sta1}
R.P.~Stanley, S.~Fomin: 
\emph{Enumerative Combinatorics, Vol.~2}, 
Cambridge University Press, Cambridge, 1999.
 
\bibitem{Sta2}
R.P.~Stanley:
\emph{Catalan Numbers}, Cambridge University Press, New York, 2015.

\bibitem{Mat}
Wolfram Research, Inc., Mathematica, Version 13.1, Champaign, IL (2022).

\end{thebibliography}
\end{document}